\documentclass[12pt, a4paper, leqno]{amsart}

\usepackage[utf8]{inputenc}
\usepackage{amsmath} 
\usepackage{amsfonts}
\usepackage{amssymb}
\usepackage{txfonts}   
\usepackage{amsfonts}     
\usepackage{graphicx}     
\usepackage[dvipsnames]{xcolor}
\usepackage[colorlinks, allcolors=RedViolet]{hyperref}

\newcommand{\R}{\varmathbb{R}}

\newcommand{\Rn}{{\varmathbb{R}^n}}
\newcommand{\Rk}{\varmathbb{R}^2}

\newcommand{\Ha}{\mathcal{H}}

\newcommand{\ve}{\varepsilon}

\def\phi{\varphi}

\let\oldmarginpar\marginpar
\renewcommand\marginpar[1]{\-\oldmarginpar[\raggedleft\footnotesize #1]%
{\raggedright\footnotesize #1}}

\usepackage{amsthm}

\swapnumbers
\theoremstyle{plain}
\newtheorem{theorem}[equation]{Theorem}
\newtheorem{lemma}[equation]{Lemma}
\newtheorem{proposition}[equation]{Proposition}
\newtheorem{corollary}[equation]{Corollary}

\theoremstyle{definition}

\newtheorem{example}[equation]{Example}

\theoremstyle{remark}
\newtheorem{remark}[equation]{Remark}

\numberwithin{equation}{section}

\title[A note on  Sobolev  inequalities in the lower limit case ]{A note on  Sobolev  inequalities \\ in the lower limit case }

\author{Petteri Harjulehto}
\address[Petteri Harjulehto]{Department of Mathematics and Statistics,
FI-00014 University of Helsinki, Finland}
\email{petteri.harjulehto@helsinki.fi}

\author{Ritva Hurri-Syrj\"anen}
\address[Ritva Hurri-Syrj\"anen]{Department of Mathematics and Statistics,
FI-00014 University of Helsinki, Finland}
\email{ritva.hurri-syrjanen@helsinki.fi}

\date{\today}

\begin{document}


\begin{abstract} 
We study Poincare-Sobolev  type inequalities for 
compactly supported smooth functions which are  defined in the Euclidean $n$-space  and whose absolute value of gradient are Choquet
$\delta /n$-integrable with respect to the $\delta$-dimensional Hausdorff content,
$n\geq 2$, $\delta\in (0,n]$.
In particular, our results imply a new Sobolev inequality for  quasicontinuous functions defined in  the Sobolev space $W^{1,1}_0(\Rn$).
As an application we extend a recently introduced superlevel Sobolev inequality into a context of  the Hausdorff content.

\end{abstract}

\maketitle

\section{Introduction}\label{Intro}

Let us recall 
the following Poincar\'e-Sobolev inequality 
 in the sense of Choquet with respect to the $\delta$-dimensional Hausdorff content $\Ha^\delta_\infty$
from \cite[Theorem 4.2 (b)]{HH-S_JFA}.
If $n\geq 2$, $\delta \in (0, n]$, $p\in(\delta/n, \delta)$, and $\kappa \in [0, 1)$, then
for every smooth function $u$ with  compact support in $\Rn$
\begin{equation}\label{classical}
\Bigg( \int_\Rn |u|^{\frac{p(\delta- \kappa p)}{\delta -p}} \, d \Ha^{\delta- \kappa p}_\infty \Bigg)^{\frac{\delta -p}{p(\delta- \kappa p)}}
\le c \Bigg( \int_\Rn |\nabla u|^p \, d \Ha^\delta_\infty \Bigg)^{\frac1p}
\end{equation}
where $c$ is a constant which depends only on $n$, $\delta$, $p$, and $\kappa$.
In particular, if  we choose $\delta =n$ and $\kappa =0$, then
the inequality \eqref{classical} reduces to the well-known  Sobolev inequality for smooth functions
with compact support which can be found in \cite[Proof of Theorem 7.10]{GT}, for example.
The exponent   $\frac{p(\delta- \kappa p)}{\delta -p}$  in \eqref{classical}
is the best possible, we refer to \cite[Example 3.14]{HH-S_JFA}.

In the present paper we study  inequalities  corresponding to the inequality \eqref{classical} when on the right-hand side 
to the absolute value of the gradient $|\nabla u| $ the exponent
$p=\delta /n$ is allowed.
We point out that
the case $\delta=n$, $p=1$, $\kappa=1$ is already solved in the famous inequality of D. R. Adams
\cite[Proposition~7, p.~122]{Adams1986}.
 This reappears also as a special case in 
 \cite[Theorem 1.7]{PonceSpector2020} where fractional derivatives are considered.
Also, the  special case  $p=\delta /n$, $\delta =n-1$, and $\kappa =0$ has been proved in \cite[Remark 4.3]{HH-S_JFA} based on Adams's inequality.
But the strong Sobolev inequality in the general limit case $p=\delta /n$ seems to be missing in  the sense of Choquet  with respect to the Hausdorff content.
Many other important  integral inequalities 
in the sense of Choquet with respect to the Hausdorff content and some other capacities  have been proven recently in
\cite{PonceSpector2020},
\cite{ChenOoiSpector2023},
\cite{ChenSpector2023},
\cite{OoiPhuc2022},
\cite{HH-S_JFA},
\cite{HCYZ},
\cite{HKST2024}, \cite{HKST2025},
\cite{HH-S_La}, \cite{HH-S_AWM},
\cite{ChenClaros2025},
\cite{Chen2025},
\cite{BCR2025},
\cite{ChenClaros2026}.

 In particular in the present paper, 
for values of $\delta$  near the dimension $n$ we obtain the following corollary from our Theorem \ref{thm:limit_case}.

\begin{corollary}\label{improvement}
Let    $n\geq 2$, $\kappa \in [0, 1]$, and $\frac{n-1}{1-\frac{\kappa }{n}} \le \delta \le n$.
Then there exists a constant $c=c(n, \delta, \kappa )$ such that
\begin{equation}\label{interesting}
\bigg ( \int_\Rn |u|^{\frac{\delta- \kappa \frac{\delta}{n}}{n-1}} \, d \Ha^{\delta- \kappa \frac{\delta}{n}}_\infty \bigg)^{\frac{n-1}{\delta- \kappa \frac{\delta}{n}}} \le
c\bigg( \int_\Rn |\nabla u|^{\frac{\delta}{n}} d \Ha^{\delta}_\infty \bigg)^{\frac{n}{\delta}}
\end{equation}
for  all smooth functions $u$ with compact support  in $\Rn$.
\end{corollary}
In Remark \ref{connections} we point out that with a choice of $\delta =n$ and $\kappa =0$ in  the inequality \eqref{interesting}
the well-known limit case of the Sobolev inequality for smooth functions  is recovered, that is the case $p=1$, which can be found for example in 
\cite[Proof for Theorem 7.10]{GT}.

In fact, we prove our results for  more general functions than smooth functions, namely  quasicontinuous Sobolev functions,
Theorem \ref{thm:new-SP-W}  and  Theorem \ref{thm:limit_case}.
Quasicontinuity is defined with respect to the Sobolev capacity \eqref{Sobolevcapacity}.
We also show that our results are essentially sharp  in Example \ref{first_example} and Example \ref{second_example}.
For readers' convenience we recall some basic properties of Choquet integral theory in Section \ref{sec:Pre} 
and  some properties of Sobolev functions in Section \ref{Limiting}.
The Poincar\'e-Sobolev inequalities are considered  in Section \ref{Limiting}, where we also compare our results to earlier known inequalities.
In Section \ref{SuperlevelSection} we give an application of our results, Theorem \ref{thm:superlevelS}.
This theorem generalises a  recently introduced 
superlevel Sobolev inequality by I.\ Kangasniemi and J.\ Onninen in \cite[Proposition~1.7]{KO}
  to a  context  the  Hausdorff content.

\section{Preliminaries}\label{sec:Pre}

In this paper the dimension of the space $\Rn$ 
is  assumed to be at least two. 

We recall the definition of the  $\delta$-dimensional Hausdorff content for any given set $E$ 
in $\Rn$,  \cite[2.10.1, p.~169]{Federer}, \cite{Adams1998, Adams2015}.
An  open ball centred at $x$ with radius $r>0$ is written as $B(x,r)$.
Suppose that $\delta \in (0, n]$.
The $\delta$-dimensional Hausdorff content of $E$ is defined by
\begin{equation}
\Ha_\infty^{\delta} (E) := \inf \bigg\{ \sum_{i=1}^\infty r_i^{\delta}: E \subset \bigcup_{i=1}^\infty B(x_i, r_i)\bigg\}\,,\label{HausdorffC}
\end{equation}
where the infimum is taken over all  
countable or finite collections of balls  that cover $E$.
The quantity \eqref{HausdorffC} is also called  the $\delta$-Hausdorff capacity.

The $\delta$-dimensional Hausdorff content 
$\Ha_\infty^{\delta}$   is an outer capacity in the sense of N.  G. Meyers \cite[p. 257]{Mey70}. 
We refer to \cite[p. 30]{Dav56} and \cite{SioS62, Adams2015}. The following property of the outer capacity is crucial:
 if $(K_i)$ is a decreasing sequence of compact sets then 
\[
\Ha_\infty^{\delta}\Big(\bigcap_{i=1}^\infty K_i \Big)
= \lim_{i \to \infty} \Ha_\infty^{\delta}(K_i).
\]

For all  Lebesgue measurable  sets $E$ in $\Rn$, the $n$-dimensional Hausdorff content  $\Ha^n_\infty (E)$ is comparable with the   Lebesgue measure of $E$ 
 which is written as $|E|$. The constants 
in the comparison inequalities
depend only on $n$.  We refer to \cite[Proposition~2.5]{HH-S_JFA} 
and \cite[Chapter~3, p.~14]{Adams1998}.

Let us recall that the $\delta$-dimensional  Hausdorff (outer) measure for a subset $E$ of  $\R^n$ is defined as
\begin{equation*}
\Ha^\delta (E) := \lim_{\rho \to 0^+}  \inf \bigg\{ \sum_{i=1}^\infty r_i^{\delta}: E \subset \bigcup_{i=1}^\infty B(x_i, r_i) \text{ and } r_i \le \rho \text{ for all } i\bigg\},
\end{equation*}
where the infimum is taken over  all  countable  or finite balls of radius at most $\rho$ that cover $E$. 
The $\delta$-dimensional Hausdorff content $\Ha^\delta_\infty$ and the $\delta$-dimensional Hausdorff measure $\Ha^\delta$ 
have the same null sets \cite[Chapter 3, p.~14]{Adams2015}.

We recall the definition of Choquet integral.
For any function $f:\Rn\to [0,\infty]$ the integral in the sense of Choquet with respect to Hausdorff content is defined by
\begin{equation*}
\int_\Rn f \, d \Ha^\delta_\infty := \int_0^\infty \Ha^\delta_\infty\big(\{x \in \Rn : f(x)>t\}\big) \, dt, 
\end{equation*}
where  on the right-hand side is the one dimensional Lebesgue integral.
Note that the Choquet integral with respect to Hausdorff content is non-linear. 
The next lemma is well known, but we include the proof for the convenience of readers.

\begin{lemma}\label{lem:same-content-everywhere}
Let $\delta \in (0, n]$, and $f, g: \Rn \to [0, \infty]$.
If $\Ha^\delta_\infty \big(\{x \in \Rn : f(x) = g(x)\} \big)=0$, then
$\int_\Rn f \,  d \Ha^{\delta}_\infty 
= \int_\Rn g \,  d \Ha^{\delta}_\infty$.
 
\end{lemma}
\begin{proof}
Since the Hausdorff content is subadditive, 
\[
\begin{split}
\Ha^{\delta}_\infty \big(\{f(x) >t\}\big) &\le \Ha^{\delta}_\infty\big(\{g(x) >t\} \cup \{f(x)\neq g(x)\}\big)\\
&\le \Ha^{\delta}_\infty\big(\{g(x) >t\}\big)  +  \Ha^{\delta}_\infty\big(\{f(x)\neq g(x)\}\big)\\ 
&= \Ha^{\delta}_\infty\big(\{g(x) >t\}\big).
\end{split}
\]
Also,  $\Ha^{\delta}_\infty\big(\{g(x) >t\}\big) \le \Ha^{\delta}_\infty\big(\{f(x) >t\}\big)$. Thus the claim follows by the definition of Choquet integral.
\end{proof}

The following lemma is due to D.\ Denneberg. 
Denneberg's result can be applied to Choquet integrals with respect to the dyadic Hausdorff content. 
Since the dyadic Hausdorff content is comparable with the ball covering  Hausdorff content $\Ha^\delta_\infty$  defined in  \eqref{HausdorffC}  by  \cite[Proposition~2.3]{YangYuan08}, we obtain the following lemma. We refer to the discussion in \cite[Theorem~3.9]{HH-S_Jia}.

\begin{lemma}\label{lem:Denneberg}
\cite[Theorem 6.3]{Den94}.
If  $\delta \in (0, n]$, then
for 
all sequences $(f_i)$ of non-negative functions $f_i:\Rn \to [0, \infty]$ 
the inequality
\begin{equation}\label{DennebergQuasiSublin}
\int_\Rn \sum_{i=1}^{\infty} f_i d \Ha^{\delta}_\infty \le c(n) \sum_{i=1}^{\infty} \int_\Rn f_i d \Ha^{\delta}_\infty\,,
\end{equation}
holds; here $c(n)$  is a constant which depends only on  the dimension $n$.
\end{lemma}

The next lemma allows 
the dimension $\delta$ in the $\delta$-dimension Hausdorff content to be changed under the integral.
This fact has been proven in  \cite[Proposition 2.3]{ChenOoiSpector2023}; we refer also to  \cite[Proposition 2.6]{HH-S_La}.   

\begin{lemma}\label{GeneralizationOV} \cite[Proposition 2.3]{ChenOoiSpector2023}.
 Let   $0 < \delta_1 < \delta_2 \le n$.
 Then the  inequality
\begin{equation*}
\biggl(\int_\Rn |f(x)|\, d \Ha_{\infty}^{\delta_2} \biggr)^{1/{\delta_2}}
\le 
\biggl(
 \frac{\delta_2}{\delta_1}\biggr)^{1/\delta_2} \biggl(\int_\Rn |f(x)|^{\frac{\delta_1}{\delta_2}} \, d \Ha^{\delta_1}_\infty \biggr)^{1/\delta_1}
\end{equation*}
holds for all functions $f: \Rn \to [-\infty, \infty]$.
\end{lemma}

For  some other basic properties of Choquet integrals we refer to  \cite{Adams1998}, \cite[Chapter 4]{Adams2015}, and \cite{CerMS11}. 
The 
Lebesgue integral defined  in $\Rn$ is denoted by $\int_\Rn f \, dx$.
 We recall that by Cavalier's principle,
\[
\int_\Rn |f| \, dx = \int_0^\infty |\{x \in \Rn : |f(x)| >t \}| \, dt.
\]
Since the $n$-dimensional Hausdorff content is comparable with the  $n$-dimensional  Lebesgue measure, 
there exist  constants $c_1(n)>0$ and $c_2(n)>0$  such that
\begin{equation}\label{compare_int}
c_1(n)  \int_\Rn |f|  \, d \Ha^{ n}_\infty \le \int_\Rn |f| \, dx \le c_2(n)  \int_\Rn |f|  \, d \Ha^{ n}_\infty
\end{equation}
for all Lebesgue measurable  functions $f:\Rn \to [-\infty, \infty]$, and  we write $ \int_\Rn |f|  \, d \Ha^{ n}_\infty \approx \int_\Rn |f| \, dx$.

\section{The Poincar\'e-Sobolev inequality}\label{Limiting}

The Poincar\'e-Sobolev inequality  
is valid also for Lipschitz continuous functions with compact support.
It is known that the exponent $\frac{p(\delta- \kappa p)}{\delta -p}$ 
in  the inequality  \eqref{PS} is the best possible in this case.
 We refer to  the result  in \cite[Theorem 4.2(b)]{HH-S_JFA} and the discussion in \cite[Example 3.14]{HH-S_JFA}.

\begin{proposition}\label{thm:SP}
Let $n\geq 2$, $\delta \in (0, n]$, $p\in(\delta/n, \delta)$, and $\kappa \in [0, 1)$. There exists a constant $c=c(n, p, \delta , \kappa )>0$ such that
\begin{equation}\label{PS}
\Bigg( \int_\Rn |u|^{\frac{p(\delta- \kappa p)}{\delta -p}} \, d \Ha^{\delta- \kappa p}_\infty \Bigg)^{\frac{\delta -p}{p(\delta- \kappa p)}}
\le c \Bigg( \int_\Rn |\nabla u|^p \, d \Ha^\delta_\infty \Bigg)^{\frac1p}
\end{equation}
for every Lipschitz continuous function $u$ with  compact support in $\Rn$.
\end{proposition}

\begin{proof}
The proof  is similar to the proof of \cite[Theorem 4.2(b)]{HH-S_JFA}. The 
difference is that
the estimate
\[
|u(x)| \le c(n) \int_\Rn \frac{|\nabla u(y)|}{|x-y|^{n-1}} \, dy \quad\text{for all} \quad x \in \Rn
\]
does not hold only for $C^1_0(\Rn)$-functions but also for all Lipschitz continuous functions with  compact support in $\Rn$
by \cite[Lemma 7.14]{GT}
\end{proof}

The  general limiting case $p=\delta/n$ in Proposition~\ref{thm:SP} 
is open,  as far as we know,   except the special cases which we have mentioned in the introduction Section \ref{Intro}. We refer to  \cite[Section~3]{HH-S_La} also. 
We are interested in an inequality which corresponds to the inequality \eqref{PS}
when $p$ is allowed  to be $\delta /n$.
First we consider the case $\delta =n$, that is $p=1$, and obtain Theorem \ref{thm:new-SP-W}.
We recall an important result of D.\ R.\ Adams. 

\begin{theorem} \cite[Proposition~7, p.~122]{Adams1986}.\label{thm:Adams}
There exists a constant $c(n)>0$ such that
\[
\int_\Rn |u| \,  d \Ha^{n-1}_\infty \le c(n) \int_\Rn |\nabla u| \, dx.
\]
for all $u \in C^{\infty}_0(\Rn)$.
\end{theorem}
 Adams's proof relies on the duality. It seems that such a proof does not  directly cover the case $p=1$ 
in an  inequality corresponding to \eqref{PS}, that is
when on  the left-hand side dimensions  
of the Hausdorff content are allowed to
be $n-\kappa$, where $\kappa \in (0, 1)$.

We wish to formulate results for more general functions, namely Sobolev functions and hence we recall some of their properties.
We write
$\|u\|_1 : =\int_{\Rn}\vert u(x)\vert\,dx$\,.
 We denote by $W^{1,1}_0(\Rn)$ the closure of $C^\infty_0(\Rn)$
with respect to the norm $u \mapsto \|u\|_1 + \|\nabla u\|_1$. 
This space is called Sobolev space.
The elements of $W^{1,1}_0(\Rn)$ are equivalence classes. 
We  identify functions whenever they agree almost everywhere  with respect to the $n$-dimensional Lebesgue measure,
 as it  is usually done.
Hence the term Sobolev functions in the space $W^{1,1}_0(\Rn)$ is justified.
In the next proof we need  the Sobolev $1$-capacity:
\begin{equation}\label{Sobolevcapacity}
\operatorname{Cap}_1 (E) := \inf \int_\Rn |\nabla u| \, dx,
\end{equation}
where the infimum is taken over all functions $u \in W^{1,1 }_0(\Rn)$ which are at least one  in  some open subset in  the set $E$. For the properties of Sobolev $1$-capacity we refer to \cite[Section 4.7]{EvaG92}. 
 Using the Sobolev $1$-capacity we define $1$-quasicontinuity, that is quasicontinuity with respect to the Sobolev $1$-capacity \eqref{Sobolevcapacity}.

We say that  a function  $u: \Rn \to [-\infty, \infty]$ is $1$-quasicontinuous, if for every $\ve>0$, there exists an open set $O$
 in $\Rn$ such that $\operatorname{Cap}_1 (O)\le \ve$ and 
 $u$ restricted to $\Rn \setminus O$ is continuous. Every $u \in W^{1,1}_0(\Rn)$ has a $1$-quasicontinuous representative, \cite[Theorem 4.19]{EvaG92}.
The $1$-quasicontinuous representative is unique in the sense that if $u^*_1$ and $u^*_2$ are two 
$1$-quasicontinuous representatives of  $u \in W^{1,1}_0(\Rn)$, then $u_1^*=u_2^*$ outside a set of zero Sobolev $1$-capacity.

We are ready to state one of our main results.

\begin{theorem}\label{thm:new-SP-W}
Let  $n\geq 2$  and 
$\kappa \in [0, 1]$. Then there exists a constant $c=c(n)$ such that
\begin{equation}\label{SPW}
\bigg(\int_\Rn |u|^{\frac{n-\kappa}{n-1}} \,  d \Ha^{n- \kappa}_\infty \bigg)^{\frac{n-1}{n-\kappa}}\le c \int_\Rn |\nabla u| \, dx
\end{equation}
for  all $1$-quasicontinuous $u \in W^{1, 1}_0(\Rn)$.
\end{theorem}

\begin{proof}
Let $\kappa \in [0, 1]$.
For  functions  $u \in C^\infty_0(\Omega)$ 
Lemma~\ref{GeneralizationOV}, Theorem~\ref{thm:Adams}, and  the  inequalities \eqref{compare_int} imply that
\begin{equation}\label{equ:Admas-corollary}
\begin{split}
\Big(\int_\Rn |u|^{\frac{n- \kappa}{n-1}} \,  d \Ha^{n-\kappa}_\infty \Big)^{\frac{n-1}{n- \kappa}}
&\le \bigg(\frac{n-\kappa}{n-1}\bigg)^{\frac{n-1}{n-\kappa}} \int_\Rn |u| \,  d \Ha^{n-1}_\infty \\
&\le  \bigg(\frac{n}{n-1}\bigg)  \int_\Rn |u| \,  d \Ha^{n-1}_\infty \le 
c(n) \int_\Rn |\nabla u| \, dx\,.
\end{split}
\end{equation}

If $u \in W^{1, 1}_0(\Rn)$,
then there exists a sequence 
of functions $\phi_i$ 
from  $C^\infty_0(\Rn)$ converging to  the function $u$ 
in the Sobolev norm $u \mapsto \|u\|_1 + \|\nabla u\|_1$.   There exists a subsequence, 
denoted by $(\phi_k)$, such that 
$(\phi_k)$ converges pointwise to a representative $u^*$ 
outside a set of zero Sobolev $1$-capacity,
that is  $1$-quasieverywhere,
moreover
$u^*$ is  the   $1$-quasicontinuous representative of $u$, we refer to \cite[Theorem 4.19 and its proof]{EvaG92}.

By \cite[Theorem 5.12, p.~220]{EvaG92} a set $E$ in $\Rn$ has a zero Sobolev $1$-capacity if and only if the set 
$E$ has a zero $(n-1)$-dimension Hausdorff measure, and thus also  the Hausdorff content $\Ha^{n-\kappa}_\infty (E) =0$ for all $\kappa \in[0, 1]$. 
We conclude that the sequence
$(\phi_k)$ converges pointwise to $u^*$ outside a set of zero $\Ha^{n-\kappa}_\infty$-content.

In \cite[Proposition~4.2]{HH-S_JGA} the authors proved Fatou's lemma for Choquet integral with respect to the dyadic Hausdorff content. Since the dyadic Hausdorff content is comparable with the ball covering  Hausdorff content $\Ha^\delta_\infty$  by  \cite[Proposition~2.3]{YangYuan08}, we obtain for every  fixed $\delta \in(0, n]$  
and  for any sequence of non-negative functions $f_i$ the inequality
\begin{equation}\label{equ:Fatou}
\int_\Rn \liminf_{i \to \infty} f_i \, d \Ha^{\delta}_\infty \le c(n) 
\liminf_{i \to \infty} \int_\Rn  f_i \, d \Ha^{\delta}_\infty.
\end{equation}
 Hence, Lemma~\ref{lem:same-content-everywhere}, the inequality \eqref{equ:Fatou} for the sequence  $(|\phi_k|)$, and
 \eqref{equ:Admas-corollary} for the function $\phi_k$ imply that
\begin{equation*}
\begin{split}
\bigg(\int_\Rn |u^*|^{\frac{n-\kappa}{n-1}} \,  d \Ha^{n- \kappa}_\infty \bigg)^{\frac{n-1}{n-\kappa}}
&= \bigg(\int_\Rn \liminf_{k \to \infty}|\phi_k|^{\frac{n-\kappa}{n-1}} \,  d \Ha^{n- \kappa}_\infty \bigg)^{\frac{n-1}{n-\kappa}}\\
&\le \bigg( c \liminf_{k \to \infty} \int_\Rn |\phi_k|^{\frac{n-\kappa}{n-1}} \,  d \Ha^{n- \kappa}_\infty \bigg)^{\frac{n-1}{n-\kappa}}\\
&\le c \liminf_{k \to \infty} \int_\Rn |\nabla \phi_k| \, dx =  \int_\Rn |\nabla u^*| \, dx. \qedhere
\end{split}
\end{equation*}
\end{proof}

We point out  that smooth functions are not necessarily dense in the 
Lebesgue type space defined by Choquet integral, we refer to \cite[Remark~6.3]{HH-S_JGA}. 
Hence, in the previous proof when
the approximating sequence comes from the definition of $W^{1, 1}_0(\Rn)$, additional work is needed to show that 
the sequence converges also  in the sense of Choquet  with respect to the Hausdorff content.

Next example shows that the exponent $\frac{n- \kappa}{n-1}$ in the inequality \eqref{SPW} is the best possible
under the circumstances stated in    Theorem~\ref{thm:new-SP-W} .         

\begin{example}\label{first_example}
 Let us choose $0<r<1$. Let $u_r: \Rk \to [0, \infty)$ be a function  
defined as
$u_r(x) =\frac1r$ in $B(0,r)$, $u_r(x)=0$ in $\Rk \setminus B(0, 2r)$, and linear in $B(0, 2r) \setminus B(0, r)$. Thus, $|\nabla u_r| \approx \frac1{r^2}$ in  $B(0, 2r) \setminus B(0, r)$, and 
$\int_{\R^2} |\nabla u_r| \, dx \approx 1$. On the other hand,
\[
\Big( \int_{\R^2} |u_r|^\alpha \, \Ha^{2-\kappa}_\infty \Big)^{\frac1\alpha} \ge \frac1r \, r^{\frac1\alpha (2-\kappa)}
= r^{\frac1\alpha(2-\kappa) -1} \to \infty
\]
as $r \to 0^+$ provided that $\frac1\alpha(2-\kappa) -1 <0$ i.e.\ if $\alpha > 2- \kappa$. Hence, 
 the requirement is
$\alpha \le 2- \kappa = \frac{2-\kappa}{2-1}$.
\end{example}

We record some surprising consequences of the previous theorem.

\begin{remark}\label{connections}
(a) 
Choosing $\kappa =0$ in Theorem~\ref{thm:new-SP-W} and using the inequalities \eqref{compare_int}
the classical Poincar\'e-Sobolev inequality 
\begin{equation}\label{Classical}
\Big(\int_\Rn |u|^{\frac{n}{n-1}} \,  dx \Big)^{\frac{n-1}{n}}\le c \int_\Rn |\nabla u| \, dx
\end{equation}
 is recovered
 for all functions  $u \in W^{1, 1}_0(\Rn)$.
This well-known inequality \eqref{Classical} was
established by  E.\ Gagliardo \cite{G} and
L.\ Nirenberg  \cite{N}.

(b) As far as we know, the case  $\kappa \in (0, 1)$ in Theorem~\ref{thm:new-SP-W} is new.
  Also, the inequality \eqref{SPW}  seems to
 be sharper  than the  inequality \eqref{Classical}. Namely, Lemma~\ref{GeneralizationOV} implies that
\[
\Big(\int_\Rn |u|^{\frac{n}{n-1}} \,  d x\Big)^{\frac{n-1}{n}}
\le 
c(n) \bigg(\int_\Rn |u|^{\frac{n-\kappa}{n-1}} \,  d \Ha^{n- \kappa}_\infty \bigg)^{\frac{n-1}{n-\kappa}}.
\]
\end{remark}

Now we prove the limiting case 
$p=\frac{\delta}{n}$ whenever
the dimension $\delta$ is close enough to $n$.
The proof uses  Adams's result  Theorem \ref{thm:Adams}

\begin{theorem}\label{thm:limit_case}
Let $n\geq 2$,  $\kappa \in [0, 1]$, and $\frac{n-1}{1-\frac{\kappa }{n}} \le \delta \le n$.
Then there exists a constant $c=c(n, \kappa, \delta)$ such that
\[
\bigg ( \int_\Rn |u|^{\frac{\delta- \kappa \frac{\delta}{n}}{n-1}} \, d \Ha^{\delta- \kappa \frac{\delta}{n}}_\infty \bigg)^{\frac{n-1}{\delta- \kappa \frac{\delta}{n}}} \le
c \bigg( \int_\Rn |\nabla u|^{\frac{\delta}{n}} d \Ha^{\delta}_\infty \bigg)^{\frac{n}{\delta}}
\]
for  all $1$- quasicontinuous  $u \in W^{1, 1}_0(\Rn)$.
\end{theorem}

\begin{proof}
 Let  $u \in W^{1, 1}_0(\Rn)$  be $1$-quasicontinuous.
 We note that $\delta- \kappa \frac{\delta}{n} \ge n-1$, since $\frac{n-1}{1-\frac{\kappa }{n}} \le \delta $.
By  Lemma~\ref{GeneralizationOV}, 
Theorem \ref{thm:Adams}, the inequalities \eqref{compare_int}, and   using once more 
Lemma~\ref{GeneralizationOV}, 
we obtain
\[
\begin{split}
\bigg ( \int_\Rn |u|^{\frac{\delta- \kappa \frac{\delta}{n}}{n-1}} \, d \Ha^{\delta- \kappa \frac{\delta}{n}}_\infty \bigg)^{\frac{n-1}{\delta- \kappa \frac{\delta}{n}}}
&\le \bigg(\frac{\delta-\kappa\frac{\delta}{n}}{n-1}\bigg)^{\frac{n-1}{\delta -\kappa\frac{\delta}{n}}} \int_\Rn |u| \, d \Ha^{n-1}_\infty \\
&\le c(n, \kappa,\delta ) \int_\Rn |\nabla u| d \Ha^{n}_\infty\\
&\le c(n, \kappa , \delta )\bigg( \int_\Rn |\nabla u|^{\frac{\delta}{n}} d \Ha^{\delta}_\infty \bigg)^{\frac{n}{\delta}}\,.
\qedhere
\end{split}
\]
\end{proof}

 Note that 
if
$\int_\Rn |\nabla u|^{\frac{\delta}{n}} d \Ha^{\delta}_\infty < \infty$, then
Lemma \ref{GeneralizationOV} 
implies  that  the function $|\nabla u|$ is integrable over $\Rn$.
Hence, the assumption 
$u \in W^{1, 1}_0(\Rn)$ in Theorem~\ref{thm:limit_case} 
is reasonable.

Next we give an example, which shows that for any fixed $\delta \in (0, n-1)$ and $p>0$ 
the inequality
\begin{equation}\label{W}
\bigg(\int_\Rn |u|^{p} \,  d \Ha^{\delta}_\infty \bigg)^{\frac{1}{p}}\le c \int_\Rn |\nabla u| \, dx
\end{equation}
 cannot hold
for all $u \in C^\infty_0(\Rn)$. Thus  the range $[0, 1]$ for $\kappa$ in Theorem~\ref{thm:new-SP-W} is 
the largest possible  range.

\begin{example}\label{second_example}
Let $\delta \in (0, n-1)$.
A Cantor type construction can give a
  compact set  $E$ in  $\Rn$ with the Hausdorff dimension $n-1$ and zero $(n-1)$-dimensional Hausdorff measure, we refer to 
  \cite[Examples 4.11 and 4.12]{Mattila}. Then the $\delta$-dimensional Hausdorff measure of  this set $E$ is infinity, and  
  hence $\Ha^\delta_\infty (E) >0$. Since the $(n-1)$-dimensional Hausdorff measure is zero, the Sobolev $1$-capacity of $E$ 
  is zero \cite[Theorem 5.12, p.~220]{EvaG92}.  Since the  set $E$ is compact, there exists a sequence $(\phi_i)$ 
  of functions
  from $C^\infty_0(\Rn)$ such that $\phi_i \ge 1$ in the set $E$ and 
\[
\int_\Rn |\nabla \phi_i| \, dx \to 0
\]
as $i \to \infty$ by \cite[Remark after Definition 4.10, p.~171]{EvaG92}. On the other hand, since $\phi_i \ge 1$ in the set $E$ 
we have for all $p>0$ that
\[
\int_\Rn |\phi_i|^{p} \,  d \Ha^{\delta}_\infty \ge \Ha^\delta_\infty (E) >0.
\] 
Hence, the inequality \eqref{W} does not hold.
\end{example}

\begin{remark}
Corollary \ref{improvement} follows from Theorem \ref{thm:limit_case} immediately, since  $C^{\infty}_0(\Rn) \subset W^{1,1}_0(\Rn)$.   
\end{remark}

\section{The superlevel Sobolev inequality}\label{SuperlevelSection}

In this chapter
 we study the superlevel Sobolev inequality  introduced  by Kangasniemi and Onninen in \cite[Proposition~1.7]{KO}.
The following lemma with 
$\kappa =0$
recovers \cite[Lemma~3.5]{KO}. 
From now on 
the essential supremum of 
 a function $u$  is written as $\|u\|_{\infty}=
\operatorname{ess\,sup}_{x\in\Rn} |u(x)|$.

\begin{lemma}\label{lem:KO3.5}
Let   $\kappa \in [0, 1]$.  There exists a constant $c>0$ such that
for all non-negative, continuous  functions $u \in W^{1,1}_0(\Rn)$, 
and all $0\le a <b < \|u\|_{\infty}$, 
the inequality
\[
|b-a| \Big(\Ha^{n- \kappa }_\infty(\{x \in \Rn :  u(x) \ge b\}) \Big)^{\frac{n -1}{n- \kappa }}
\le c \int_{\{x \in \Rn : a<u(x)<b\}} |\nabla u| \, dx
\]
is valid.
\end{lemma}

\begin{proof}
Let us 
consider a function
$\psi: \Rn \to \R$ which we define as
\[
\psi(x) :=
\begin{cases}
0, &\text{ if } u(x) \le a;\\
u(x) -a, &\text{ if } a< u(x) < b;\\
b- a, &\text{ if } u(x) \ge b.\\
\end{cases}
\]
Thus,  $\psi \in W^{1,1}_0(\Rn)$, since $u \in W^{1, 1}_{0}(\Rn)$.
Moreover $|\psi(x)| \le |u(x)|$ in $\Rn$,  $|\nabla \psi| = |\nabla u|$ a.e.\ in $\{ a <u(x) <b\}$, and
$|\nabla \psi| =0$ a.e.\ in $\{ u(x) \le a \text{ or } u(x) \ge b\}$.
We  apply Theorem~\ref{thm:new-SP-W}  to  the function $\psi$ and obtain
\[
\begin{split}
&|b-a| \Big(\Ha^{n- \kappa }_\infty(\{ u(x) \ge b\}) \Big)^{\frac{n -1}{n- \kappa }}
= \Big( \int_{\{ u(x) \ge b\}} |b-a|^{\frac{n- \kappa }{n -1}} \, d\Ha^{n-\kappa }_\infty \Big)^{\frac{n -1}{n-\kappa }}\\
&\quad \le \Big( \int_{\Rn} |\psi|^{\frac{n- \kappa}{n -1}} \, d\Ha^{n-\kappa }_\infty \Big)^{\frac{n -1}{n-\kappa }}
\le c \int_{\Rn} |\nabla \psi| \, dx
=c \int_{\{a<u(x)<b\}} |\nabla u| \, dx. \qedhere
\end{split}
\]
\end{proof}

 We recall a  staircase lemma from \cite[Lemma 2.1]{KO}.

\begin{lemma}\label{lem:Staircase}
Let $F:[0, \infty] \to [0, \infty]$ be a left-continuous and non-decreasing function.  Let $s:=\sup\{t \in [0, \infty] : F(t) < F(\infty)\}$.
If $s>0$, then for every $\ve>0$ there exists an increasing sequence $t_0<t_1<t_2<\ldots$ of elements of $[0, s)$  with $t_0=0$, 
$\lim_{i \to \infty}t_i =s$ such that 
 the inequality $|F(t_i) - F(t)| \le \ve$
holds for all $t \in (t_{i-1}, t_i]$, 
$ i=1, 2,\dots $.
\end{lemma}

Kangasniemi and Onninen proved the superlevel Sobolev inequality  with Lebesgue measure in \cite[Proposition~1.7]{KO}. 
Their result corresponds to  the case  $\kappa =0$ in  Theorem~\ref{thm:superlevelS}. Following their ideas we generalise the superlevel Sobolev inequality to the  case for  Hausdorff content with lower dimensions.

\begin{theorem}\label{thm:superlevelS}
Let $\kappa \in [0, 1]$. Then there exists a constant $c>0$ such that
\begin{equation}\label{SuperlevelChoquet}
\|u\|_{\infty} \le c  \int_{\Rn} \frac{|\nabla u (x)|}{\Big(\Ha^{n- \kappa}_\infty \big(\{y \in \Rn : u(y) \ge u(x)\}\big) \Big)^{\frac{n -1}{n- \kappa}}} \, dx
\end{equation}
for all non-negative, continuous  functions  $u \in W^{1, 1}_0(\Rn)$ with compact support in $\Rn$.
\end{theorem}

\begin{proof}
We may  assume that $\|u\|_\infty >0$. Let us write
\[
F(t):= \frac{1}{\Big(\Ha^{n- \kappa}_\infty(\{u(y) \ge t\}) \Big)^{\frac{n -1}{n- \kappa}}}, \quad t\geq 0.
\]
Thus, $F(t) =\infty$ if $t > \|u\|_\infty$, and $F(t) \in (0, \infty)$ if $0\le t < \|u\|_\infty$.
Let $t$ be fixed.
Since $u $ is a continuous function  with a compact support,  the set  $\{ u(y) \ge t\}$ is closed and bounded i.e.\ compact. 
Hence for any increasing sequence $(t_i)$ converging to $t$,  the  sets $(\{u(y) \ge t_i\})_i$  form  a decreasing sequence of compact sets. Thus,
\[
\lim_{i \to \infty} \Ha^{n- \kappa p}_\infty(\{u(y) \ge t_i\}) = 
\Ha^{n- \kappa p}_\infty\bigg(\bigcap_i \{u(y) \ge t_i\} \bigg) 
=  \Ha^{n- \kappa p}_\infty( \{u(y) \ge t\}),
\] 
where the last inequality follows since $t_i \nearrow t$. Hence, the function  $F$ is left-continuous. 
Since the Hausdorff content is monotone, 
the inequality
$\Ha^{n- \kappa p}_\infty(\{u(y) \ge t_2\}) \le \Ha^{n- \kappa p}_\infty(\{u(y) \ge t_1\})$  holds whenever $t_1<t_2$ and thus $F(t_1) \le F(t_2)$ i.e. 
the function
$F$ is non-decreasing. 
Let us choose $s:= \|u\|_\infty$   in  Lemma \ref{lem:Staircase}  and let $\epsilon >0$ be given.
Hence,
there exists a strictly increasing sequence $(t_i)$ converging to $\|u\|_\infty$ such that $t_0=0$ and $F(t_{i}) \le F(t) + \ve$ for all $t \in (t_{i-1}, t_i]$, { $i=1,2,\dots $.

Lemma~\ref{lem:KO3.5}  yields  that
\[
\|u\|_\infty = \sum_{i=0}^\infty (t_{i+1} - t_i)
\le \sum_{i=0}^\infty c F(t_{i+1}) \int_{\{t_i<u(x)<t_{i+1}\}} |\nabla u| \, dx. 
\]
By Lemma~\ref{lem:Staircase}
for the sets
$\{t_i<u(x)\le t_{i+1}\}$}
we have $F(t_{i+1}) \le F(u(x)) + \ve$.
Thus,
\[
\|u\|_\infty 
\le c \sum_{i=0}^\infty   \int_{\{ t_i<u(x)\le t_{i+1}\}} (F(u(x)) + \ve) |\nabla u| \, dx.
\]
Since the sets $\{t_i<u(x)\le t_{i+1}\}$ , i=1,2,\dots , are disjoint, the properties of Lebesgue integral yield that
\[
\|u\|_\infty \le c   \int_{\Rn} F(u(x))|\nabla u| \, dx + \ve \int_\Rn |\nabla u| \, dx.
\]
Since $\int_\Rn |\nabla u| \, dx <\infty$, the claim  of the theorem follows by letting $\ve \to 0^+$.
\end{proof}

We recall that there are 
non-negative, continuous  functions  $u \in W^{1, 1}_0(\Rn)$  which do not have compact support in $\Rn$, for example $u:\Rn\to\R$,  $u(x)=\exp{(-\vert x\vert ^n)}$.

We  note  that the  nominator $\Ha^{n- \kappa}_\infty \big(\{y \in \Rn : u(y) \ge u(x)\}\big)$ in Theorem~\ref{thm:superlevelS} is positive outside a set of zero  $\Ha^{n-\kappa}_\infty$-content.

\begin{remark} 
 Let $x_0\in\Rn$ be fixed.
If $u(x_0)< \|u\|_\infty$, then
 there exits a ball centred at $x_0$ in the set
$\{y \in \Rn : u(y) \ge u(x_0)\}$ 
by the continuity of 
the function $u$.
Hence  $\Ha^{n- \kappa}_\infty \big(\{y \in \Rn : u(y) \ge u(x_0)\}\big)>0$. 
If we write  $E:= \{x\in \Rn: u(x) = \|u\|_\infty\}$  and if   $\Ha^{n-\kappa}_\infty(E)>0$, then the nominator $\Ha^{n- \kappa}_\infty \big(\{y \in \Rn : u(y) \ge u(x)\}\big)$  is positive everywhere. 
If $\Ha^{n-\kappa}_\infty(E)=0$, then the nominator $\Ha^{n- \kappa}_\infty \big(\{y \in \Rn : u(y) \ge u(x)\}\big)$  is positive  outside a set of zero  $\Ha^{n-\kappa}_\infty$-content.
\end{remark}

We show that the right-hand side of the inequality \eqref{SuperlevelChoquet}  in Theorem~\ref{thm:superlevelS} can be replaced by a lower dimensional Choquet integral, but the result might be weaker than the original one
\eqref{SuperlevelChoquet}.

\begin{remark} 
Applying the inequalities  \eqref{compare_int} and Lemma~\ref{GeneralizationOV} to the gradient part on the right-hand side of the inequality \eqref{SuperlevelChoquet} in
Theorem~\ref{thm:superlevelS} yields that
\[
\|u\|_{\infty} \le c  \bigg(\int_{\Rn} \frac{|\nabla u (x)|^{\frac{n-\kappa}{n}}}{\Big(\Ha^{n- \kappa}_\infty \big(\{y \in \Rn : u(y) \ge u(x)\}\big) \Big)^{\frac{n -1}{n}}} \, \Ha^{n-\kappa}_\infty\bigg)^{\frac{n}{n-\kappa}}
\]
for all non-negative, continuous  functions  $u \in W^{1, 1}_0(\Rn)$ with compact support in $\Rn$.\\
\end{remark}

 Next we compare  the inequality \eqref{SuperlevelChoquet}  in Theorem~\ref{thm:superlevelS} to  the theorem by Kangasniemi and Onninen, \cite[Proposition~1.7]{KO}.
 
\begin{remark} 
Let $0<\delta_1<\delta_2 \le n$.
 Let $E$ be a set in  $\Rn$ and  $\ve >0$.
 We pick up a ball covering $(B(x_i, r_i))$  according to  the definition of $\Ha^{\delta_1}_\infty(E)$ such that 
$\sum_i r_i^{\delta_1}<\Ha^{\delta_1}_\infty(E) + \ve$. Since $t \mapsto t^{\delta_1/\delta_2}$ is strictly concave, we have
\[
\Big(\Ha^{\delta_2}(E) \Big)^{\frac{\delta_1}{\delta_2}} \le \Big( \sum_i r_i^{\delta_2}\Big)^{\frac{\delta_1}{\delta_2}} < \sum_i r_i^{\delta_1} \le \Ha^{\delta_1}_\infty(E) + \ve.
\]
 Letting $\ve \to 0$ implies  that $\Ha^{\delta_2}(E)^{\frac1{\delta_2}}  \le \Ha^{\delta_1}_\infty(E)^{\frac1{\delta_1}}$.
 Then,
 choosing $\delta_1= n- \kappa$ and $\delta_2=n$ yields for every $\kappa \in [0, 1]$ that
\[
\begin{split}
&\int_{\Rn} \frac{|\nabla u (x)|}{\Big(\Ha^{n- \kappa}_\infty \big(\{y \in \Rn : u(y) \ge u(x)\}\big) \Big)^{\frac{n -1}{n- \kappa}}} \, dx \\
&\quad \le \int_{\Rn} \frac{|\nabla u (x)|}{\Big(\Ha^{n}_\infty \big(\{y \in \Rn : u(y) \ge u(x)\}\big) \Big)^{\frac{n -1}{n}}} \, dx\\
&\qquad\approx \int_{\Rn} \frac{|\nabla u (x)|}{|\{y \in \Rn : u(y) \ge u(x)\}\big)|^{\frac{n -1}{n}}} \, dx.
\end{split}
\]
This means  that the inequality \eqref{SuperlevelChoquet} in Theorem~\ref{thm:superlevelS} 
with $\kappa\in (0,1]$
might be sharper than the theorem by Kangasniemi and Onninen, \cite[Proposition~1.7]{KO} which corresponds to the case $\kappa =0$.
\end{remark}


%
%
%
%
%

\bibliographystyle{amsalpha}

\end{document}